\documentclass{amsart}[12pt]

\usepackage{amsfonts,amssymb,amsmath,amscd,amsthm, mathrsfs}
\usepackage{graphicx}
\usepackage{epstopdf}
\usepackage{hyperref}
\usepackage{upgreek}

\newtheorem{thm}{Theorem}[section]
\newtheorem{prop}{Proposition}[section]
\newtheorem{lm}{Lemma}[section]
\newtheorem{cor}{Corollary}[section]

\theoremstyle{definition}
\newtheorem{dfn}{Definition}[section]

\theoremstyle{remark}
\newtheorem{rem}{Remark}[section]

\newcommand{\rr}{\mathbb{R}}
\newcommand{\zz}{\mathbb{Z}}

\newcommand{\ey}{\frac{1}{2}}

\newcommand{\A}{\mathcal{A}}

\newcommand{\aep}{\mathcal{A}^{\ep}}
\newcommand{\aen}{\mathcal{A}^{\ep_n}}
\newcommand{\F}{\mathcal{F}}
\newcommand{\U}{\mathfrak{U}}

\newcommand{\e}{\mathbf{e}}

\newcommand{\I}{\mathbf{I}}

\newcommand{\ic}{\mathcal{I}}
\newcommand{\jc}{\mathcal{J}}

\newcommand{\X}{\mathcal{X}}

\newcommand{\N}{\mathbf{N}}

\newcommand{\al}{\alpha}

\newcommand{\dl}{\delta}

\newcommand{\ep}{\varepsilon}

\newcommand{\Lmd}{\Lambda}
\newcommand{\lmd}{\lambda}
\newcommand{\om}{\omega}

\newcommand{\zt}{\zeta}

\newcommand{\en}{\ep_n}
\newcommand{\dn}{\dl_n}

\newcommand{\qd}{\dot{q}}

\newcommand{\qt}{\tilde{q}}

\newcommand{\qn}{q^n}

\newcommand{\xin}{\xi^n}
\newcommand{\etn}{\eta^n}
\newcommand{\ztn}{\zt^n}
\newcommand{\xih}{\hat{\xi}}
\newcommand{\zth}{\hat{\zt}}

\newcommand{\mo}{m^{-}}

\newcommand{\ut}{\widetilde{U}}
\newcommand{\fut}{\widetilde{\U}}

\begin{document}

	\title[Application of Morse index]{Application of Morse index in weak force $N$-body problem}
	\author{Guowei Yu}\thanks{The author acknowledges the support of the ERC Advanced Grant 2013 No.  339958 ``Complex Patterns for Strongly Interacting Dynamical Systems - COMPAT''} 

    \email{guowei.yu@unito.it}

 	\address{Dipartimento di Matematica ``Giuseppe Peano'', Universit\`a degli Studi di Torino, Italy}

	\begin{abstract} Due to collision singularities, the Lagrange action functional of the N-body problem in general is not differentiable. Because of this, the usual critical point theory can not be applied to this problem directly. Following ideas from \cite{BR91}, \cite{Tn93a} and \cite{ABT06}, we introduce a notion called weak critical point for such an action functional, as a generalization of the usual critical point. A corresponding definition of Morse index for such a weak critical point will also be given. Moreover it will be shown that the Morse index gives an upper bound of the number of possible binary collisions in a weak critical point of the $N$-body problem with weak force potentials including the Newtonian potential. 
	\end{abstract}
	
	\maketitle
	
\section{Introduction} \label{sec:intro}

The motion of $N$ point masses, $m_i>0, i \in \N:=\{1, \dots, N\}$, under the universal gravitational force is a classic problem that has been studied by many authors since the time of Newton. Let $q_i \in \rr^d$ be the position of mass $m_i$, then it satisfies the following equation
\begin{equation}
\label{eq: N body} m_i \ddot{q}_i = \frac{\partial U(q)}{\partial q_i} = - \al \sum_{j \in \N \setminus \{i\}}   m_i m_j\frac{q_i -q_j}{|q_i-q_j|^{\al+2}}, \;\; \forall i \in \N,
\end{equation}
where $q=(q_i)_{i \in \N} \in \rr^{dN}$ and $U(q)= \sum_{\{i < j\} \subset \N} \frac{m_im_j}{|q_i-q_j|^{\al}}$ is the potential function (the negative potential energy). $\al$ here is a positive constant.

Traditionally $U$ is called a \emph{strong force potential}, when $\al \ge 2$, and a \emph{weak force potential}, when $0 < \al <2$. The Newtonian potential is a weak force potential corresponding to $\al=1$. 

This is a singular Lagrange system with the Lagrangian
$$ L(q, \qd):= K(\qd) + U(q), \; \text{ where } K(\qd) = \ey \sum_{i \in \N} m_i |\qd_i|^2.$$
The singularities are caused by collisions between two or more masses
\begin{equation}
\label{eq: collision} \Delta := \{q \in \rr^{dN}: \; q_i = q_j, \text{ for some } \{i \ne j \} \subset \N \}. 
\end{equation}
For any $\I \subset \N$ with $|\I| \ge 2$ ($|\I|$ is the cardinality of $\I$), we say $q$ has an \emph{$\I$-cluster collision} at the moment $t$, when 
\begin{equation*} \forall i \in \I, \;\; 
\begin{cases}
 q_i(t)= q_j(t), \; & \text{ if } j \in \I, \\
 q_i(t) \ne q_j(t), \; & \text{ if } j \in \N \setminus \I,
\end{cases}
 \end{equation*}
An $\I$-cluster collision is a \emph{binary collision}, when $|\I|=2$. 

Let $H^1([T_1, T_2], \rr^{dN})$ be the space of Sobolev paths defined on $[T_1, T_2]$, and $\hat{\rr}^{dN} := \rr^{dN} \setminus \Delta$ the set of collision-free configurations, we say a path $q \in H^1([T_1, T_2], \rr^{dN})$ is \emph{collision-free}, if $q(t) \in \hat{\rr}^{dN}$, for any $t \in [T_1, T_2]$. It is well known the Lagrange action functional 
\begin{equation}
\label{eq: action functional} \A(q; T_1, T_2):= \int_{T_1}^{T_2} L(q, \qd) \,dt, \;\; \forall q \in H^1([T_1, T_2], \rr^{dN}),
\end{equation}
is $C^2$ on $H^1([T_1, T_2], \hat{\rr}^{dN})$ (see \cite{AZ93}), and a critical point of $\A$ in $H^1([T_1, T_2], \hat{\rr}^{dN})$ is a classical solution of \eqref{eq: N body}. 

The solution of \eqref{eq: N body} is invariant under linear translations, in many cases it will be more convenient to fix the center of mass at the origin, so we set
$$ \X:=\{ q \in \rr^{dN}: \; \sum_{i \in \N} m_i q_i = 0 \}, \;\; \hat{\X} : = \X \setminus \Delta. $$
As the action functional is also invariant under linear translation, a critical point of $\A$ in $H^1([T_1, T_2], \hat{\X})$ will be a classical solution of \eqref{eq: N body} as well.

In general it is much easier to apply variational methods to the $N$-body, when the potential is a \emph{strong force}, i.e. $\al \ge 2$, as in this case any path with a finite action value must be collision-free, see \cite{CGMS02}. It is not so, when the potential is a \emph{weak force}, i.e. $\al \in (0, 2)$, as the attracting force between the masses are too weak and the action value of a path with collision may still be finite, see \cite{Go77}. 

Because of this, for the strong force $N$-body problem, many results have been obtained by different authors using both minimization and non-minimization variational methods, see \cite{BR91}, \cite{AZ93}, \cite{MT93a}, \cite{MT93b}, \cite{Mo98}, \cite{CGMS02} and the references within. However the problem is much more difficult for weak force potentials due to the possibility of collision. Set back by this, Bahri and Rabinowitz introduced the so called \emph{generalized solution} in \cite{BR89} and \cite{BR91} (see Definition \ref{dfn: generalized solution}), where such a solution is allowed to have a non-empty set of collision moments with zero Lebesgue measure. Here we are only interested in the weak force $N$-body problem, so we assume $\al \in (0, 2)$ in the rest of the paper.

For action minimization methods, the breakthrough followed the proofs of the Figure-Eight solution of the three body by Chenciner and Montgomery \cite{CM00} and the Hip-Hop solution of the four body by Chenciner and Venturelli \cite{CV00}, where both solutions are found as collision-free minimizers of the action functional under proper symmetric constraints. Since then action minimization methods have thrived in the study of the $N$-body problem with Newtonian potential as well as other weak force potentials. We refer the interesting readers to \cite{C02}, \cite{FT04}, \cite{Ch08}, \cite{Y15b}, \cite{Yu17} and the references within. 

Now it is more or less well understood, when we can show an action minimizer is collision-free. 
\begin{dfn}
\label{dfn: local minimizer} $q \in H^1([T_1, T_2], \X)$ is a \textbf{local action minimizer of $\A$ with fixed-ends} in $H^1([T_1, T_2], \X)$ , if there is a $\dl>0$ small enough, such that 
$$ \A(q; T_1, T_2) \le \A(q+\qt; T_1, T_2), \;\; \forall \qt \in H^1_0([T_1, T_2], \X) \text{ with }  \|\qt\|_{H^1}\le \dl,$$
where $H^1_0([T_1, T_2], \X)$ is the space of Sobolev paths with compact support in $[T_1, T_2]$. 
\end{dfn}
The following fundamental result is due to Marchal \cite{Mc02} and Chenciner \cite{C02}, when $\al=1$, and Ferrario and Terracini \cite{FT04}, when $\al \in (0, 2)$. 
\begin{thm}
\label{thm: fixed ends minimizer} For any $\al \in (0, 2)$, when $d \ge 2$, if $q$ is a local action minimizer of $\A$ with fixed-ends in $H^1([T_1, T_2], \X)$, then $q(t)$ is collision-free, for any $t \in (T_1, T_2)$. 
\end{thm}

Despite of the above progress, by our knowledge, for the $N$-body problem no result seems to be available regarding how to rule out collision when the corresponding path is obtained through non-minimization methods, like minimax or mountain-pass. Meanwhile for a special type of singular Lagrange systems with weak force potentials (essentially equivalent to perturbations of the $N$ center problem), using \emph{Morse index theory}, in a series of papers (\cite{Tn93a}, \cite{Tn93b}, \cite{Tn94a}, \cite{Tn94b}), Tanaka showed how to rule out collisions when a critical point was obtained by the minimax approach of Bahri and Rabinowitz \cite{BR89}. The main purpose of our paper is to generalize Tanaka's idea to the $N$-body problem and show that the Morse index of a critical point can be used to give an upper bound of the number of binary collisions that could occur in it. 


To give the precise statement, we recall the following definition according to \cite{ABT06}. 
\begin{dfn}
\label{dfn: morse index}  Let $X$ be a Hilbert space and $\Omega$ an open and dense subspace of it, i.e. $\bar{\Omega} = X$. If a functional $\F$ is lower semi-continuous in $X$ and $C^2$ in $\Omega$, then $c \in \rr$ will be called a \textbf{critical value} of $\F$, if there is $q \in \Omega$, such that $\F(q)=c$ and the first derivative of $\F$ vanishes at $q$, i.e. $d \F(q)=0$. Moreover such a $q$ will be called a \textbf{critical point} of $\F$.

If $q \in \Omega$ is a critical point of $\F$, we define its \textbf{Morse index} (with respect to $\F$), $m^-_{\Omega}(q, \F)$, as the dimension of the largest subspace of $\Omega$, where the second derivative $d^2 \F(q)$ is negative definite. 
\end{dfn}

The above definition suits the study of the $N$-body problem, as $H^1([T_1, T_2], \hat{\X})$ is an open and dense subspace of $H^1([T_1, T_2], \X)$ with the action functional $\A$ being $C^2$ in $H^1([T_1, T_2], \hat{\X})$ and lower semi-continuous in $H^1([T_1, T_2], \X)$. Since $\A$ is generally not differentiable at a collision path, such a path can not be  a \emph{critical point} and moreover it does not have a well-defined Morse index for a collision path. Although a collision path can still be \emph{a local minimizer}. 

To deal with the above problem, following ideas from \cite{BR89}, \cite{Tn93a} and \cite{ABT06}, let's perturb the weak force $N$-body problem \eqref{eq: N body} by a strong force potential
\begin{equation}
\label{eq; V} \ep \U (q) := \ep \sum_{ \{ i < j \} \subset \N} \frac{m_i m_j}{|q_i -q_j|^2}, \; \text{ for } \; \ep > 0 \text{ small enough}. 
\end{equation}
Then the motion of masses satisfies
\begin{equation}
  \label{eq: N body strong} m_i \ddot{q}_i = - \al \sum_{ j \in \N \setminus \{i\}}  \frac{ m_i m_j(q_i -q_j)}{|q_i-q_j|^{\al+2}} - 2 \ep \sum_{ j \in \N \setminus \{i\}}  \frac{ m_i m_j(q_i -q_j)}{|q_i-q_j|^4} , \; \forall i \in \N,
  \end{equation}  
which is the Euler-Lagrange equation of the action functional 
\begin{equation}
\label{eq: action fun strong perturb} \aep(q; T_1, T_2) : = \int_{T_1}^{T_2} L^{\ep}(q, \qd) \,dt, \; \text{ where } \; L^{\ep}(q, \qd):= L(q, \qd)+ \ep \U(q).
\end{equation} 
Furthermore we set $\A^{0}(q; T_1, T_2) = \A(q; T_1, T_2)$. When $\ep>0$, any path with a finite action value of $\aep$ must be collision-free. 

Given an arbitrary path $q \in H^1([T_1, T_2], \X)$, set 
$$ H^1(q): = \{ \qt \in H^1([T_1, T_2], \X): \; \qt(T_i) = q(T_i), i = 1, 2 \}; $$
$$ \hat{H}^1(q): = H^1(q) \cap H^1([T_1, T_2], \hat{\X} \}. $$
Then $\hat{H}^1(q)$ is an open and dense subset of $H^1(q)$, where $\aep$ is lower semi-continuous in $H^1(q)$ and $C^2$ in  $\hat{H}^1(q)$. Hence if $q$ is collision-free and $d \A^{\ep}(q)=0$, then it is a critical point of $\A^{\ep}$, and we will denote its Morse index in $\hat{H}^1(q)$ by $m^-_{T_1, T_2}(q, \aep)$. Now we introduce a notion called \emph{weak critical points} as a generalization of the usual critical points. 

\begin{dfn}
\label{dfn: weak critical pt Morse index} We say a path $q \in H^1([T_1, T_2], \X)$ (which may contain collision) with finite action value, $\A(q; T_1, T_2) < \infty$, is a \textbf{weak critical point} of $\A$, if there exists a sequence of positive numbers $\en \to 0$ and a sequence of $\qn \in H^1([T_1, T_2], \X)$, such that
\begin{enumerate}
\item[(i).] $\aen(\qn; T_1, T_2) < C$, $\forall n $, for some finite constant $C$;
\item[(ii).] $\qn$ is a critical point of $\aen$, for any $n$; 
\item[(iii).] $\qn \to q$ weakly in $H^1$-norm and strongly in $L^{\infty}$-norm.
\end{enumerate}

$c=\A(q; T_1, T_2)$ will be called a \textbf{weak critical value} of $\A$, and the \emph{Morse index} of such a weak critical point $q$ in $H^1(q)$ (with respect to $\A$) will be defined as
\begin{equation}
\label{eq: morse index A} \mo_{T_1, T_2}(q, \A) = \inf \liminf_{n \to \infty} \mo_{T_1, T_2}(\qn, \aen),
\end{equation}
where the infimum is taken over all sequences $\en$ and $\qn$ satisfying the above conditions.  
\end{dfn}
\begin{rem}
A similar notation was introduced in \cite{ABT06}, where it was called \textbf{generalized critical point}. 
\end{rem}


With the above definition, we have the following result, which can be seen as a partial generalization of Theorem \ref{thm: fixed ends minimizer}. 
\begin{thm}
\label{thm: fix end not minimizer} When $d \ge 3$, given a weak critical point $q \in H^1([T_1, T_2], \X)$ of $\A$, let $\mathcal{B}(q)$ represent the number of binary collisions occurring in $q(t)$, $t \in (T_1, T_2)$ (when there are more than one binary collision at a given moment, each of them should be counted separately), then
\begin{equation}
  \label{eq; bound number of binary coll} (d-2)i(\al)\mathcal{B}(q) \le \mo_{T_1, T_2}(q, \A),
 \end{equation}
where 
\begin{equation}
\label{eq: i al} i(\al) = \max\{k \in \zz: \; k < \frac{2}{2-\al} \}.
\end{equation}
In particular $q(t), t \in (T_1, T_2)$, is free of binary collision, i.e. $\mathcal{B}(q)=0$, if
$$ m^-_{T_1, T_2}(q, \A) < (d-2) i(\al). $$
\end{thm}
\begin{rem}
\label{rem; i al} Notice that by \eqref{eq: i al}, $i(\al)=1$, if $\al \in (0, 1]$ and $i(\al) \ge 2$, if $\al \in (1, 2)$. Moreover $i(\al)$ goes to infinity, as $\al$ goes to $2$. 
\end{rem}

\begin{rem}
It seems the above result is the best we can get based on Tanaka's idea. In particular, we are unable to obtain any nontrivial result, when $d=2$. For an explanation see Remark \ref{rem d=2 no result}. 
\end{rem}

The idea of using Morse index to rule out collision should work even when a collision cluster has more than two masses, although the technical difficulty seems very challenging. This is because when two masses approach to a binary collision, they behaves more and more like the two body problem, where the solutions are well understood and their Morse indices are relatively easy to compute. However when the collision cluster has more than two masses, as they approach to collision, the dynamics is much more complicate (see \cite{Mg74}) and the computation of Morse indices of the relevant solutions is also much more difficult. Despite of this, some progresses have been made recently in \cite{BS08}, \cite{BHPT17} and \cite{HY17}.

Theorem \ref{thm: fix end not minimizer} has the following obvious corollaries.
\begin{cor} \label{cor: index 0}
When $\mo_{T_1, T_2}(q, \A) =0$ and $d \ge 3$, $q(t)$, $t \in (T_1, T_2),$ is free of binary collision.
\end{cor}

\begin{cor} \label{cor: index 1}
When $\mo_{T_1, T_2}(q, \A) = 1$, the following results hold.
\begin{enumerate}
\item[(a).] If $d \ge 4$ and $\al \in (0, 2)$, then $q(t)$, $t \in (T_1, T_2)$ is free of binary collision.
\item[(b).] If $d = 3$ and $\al \in (1, 2)$, then $q(t)$, $t \in (T_1, T_2)$ is free of binary collision.
\item[(c).] If $d = 3$ and $\al \in (0, 1]$, then $q(t)$, $t \in (T_1, T_2)$ has at most one binary collision, i.e. $\mathcal{B}(q) \le 1$.
\end{enumerate}
\end{cor}

Notice that in Corollary \ref{cor: index 1}, when $d=3$ and $\al \in (0, 1]$, the weak critical point may very well contains a binary collision and in this case we have the following result.

\begin{thm}
\label{thm: regularization} When $d=3$ and $\al=1$, let $q \in H^1([T_1, T_2], \X)$ be a weak critical point of $\A$ with $m^-_{T_1, T_2}(q, \A)=1$, if there is a binary collision between $m_{i_1}$ and $m_{i_2}$ at a moment $t_0 \in (T_1, T_2)$, then both limits 
$\lim_{t \to t_0^{\pm}} \frac{q_{i_2}(t)-q_{i_1}(t)}{|q_{i_2}(t)-q_{i_1}(t)|}$ exist and equal to each other. 
\end{thm}

\begin{rem}
The above result is interesting, because it is well-known if a solution of the spatial $N$-body problem has a single binary collision at a moment (no other partial collision exists at the same moment), then it can be regularized by Kustaanheimo-Stiefel regularization \cite{KS65}. With the result from the above theorem, under the assumption that there is no other partial collisions, one can show the generalize solution corresponding to the weak critical point is actually a classical solution in the regularized system. 
\end{rem}

\begin{rem}
Results similar to Theorem \ref{thm: regularization} can be obtained for potentials with $\al \ne 1$, see Lemma \ref{lem: asymp angle}. However the problem of regularizing a binary collision is more complicate for non-Newtonian potentials, see \cite{MG81}. 
\end{rem}

Since the Morse index of a critical point obtained by the mountain pass theorem must be less than or equal to one (see \cite{Hf85}), we believe Corollary \ref{cor: index 0} and \ref{cor: index 1} and Theorem \ref{thm: regularization} could be useful, when mountain pass methods are used in the study the $N$-body problem. This shall be discussed in a forthcoming paper.  

Our paper is organized as follows: in Section \ref{sec: generalized} we show a weak critical point is a generalized solution, in Section \ref{sec: Tanaka} the proofs of the main results will be given, and in Section \ref{sec: lem 1} and \ref{sec: lm asymp angle} we give the proofs of some technical lemmas.





\section{Generalized solutions} \label{sec: generalized} 
Consider the perturbed $N$-body problem \eqref{eq: N body strong}, for any subset of indices $\I \subset \N$, we define the \emph{Lagrangian} and \emph{energy} of the $\I$-cluster as 
$$ L^{\ep}_{\I}(q, \qd): = K_{\I}(\qd)+U_{\I}(q)+\ep \U_{\I}(q),$$
$$ E^{\ep}_{\I}(q, \qd):= K_{\I}(\qd)-U_{\I}(q)-\ep \U_{\I}(q), $$
where
$$ K_{\I}(\qd):= \ey \sum_{i \in \I} m_i|\qd_i|^2, \;\; U_{\I}(q): = \sum_{\{i<j\}\subset \I}\frac{m_i m_j}{|q_i-q_j|^{\al}}, \; \; \U_{\I}(q): = \sum_{\{i < j \} \subset \I} \frac{m_i m_j}{|q_i-q_j|^2}. $$

Let $\I':= \N \setminus \I$ denote the complement of $\I$ in $\N$, then 
$$ L^{\ep}(q, \qd) = L^{\ep}_{\I}(q, \qd) + L^{\ep}_{\I'}(q, \qd) + U_{\I, \I'}(q) +\ep \U_{\I, \I'}(q), $$
$$ E^{\ep}(q, \qd) = E^{\ep}_{\I}(q, \qd) + E^{\ep}_{\I'}(q, \qd) - U_{\I, \I'}(q) - \ep \U_{\I, \I'}(q), $$
where 
$$ U_{\I, \I'}(q): = \sum_{i \in \I, j \in \I'} \frac{m_im_j}{|q_i -q_j|^{\al}}, \; \; \U_{\I, \I'}(q): = \sum_{i \in \I, j \in \I'} \frac{m_im_j}{|q_i -q_j|^{2}}. $$

\begin{dfn} \label{dfn: generalized solution} $q \in H^1([T_1, T_2], \X)$ is a \textbf{generalized solution} of \eqref{eq: N body}, if it satisfies the following conditions:
\begin{enumerate}
\item[(i).] $\Delta^{-1}(q): = \{t \in [T_1, T_2]: \; q(t) \in \Delta\}$ has measure $0$ in $[T_1, T_2]$; 
\item [(ii).] $q \in C^2$ on $[T_1, T_2] \setminus \Delta^{-1}(q)$ and satisfies \eqref{eq: N body};
\item [(iii).] the total energy of $q(t)$, $E(t)=E(q(t), \qd(t))$,  is a constant, for all  $t \in [T_1, T_2] \setminus \Delta^{-1}(q)$;
\item [(iv).] for any subset $\I \subset \N$ and sub-interval $(t_1, t_2) \subset [T_1, T_2]$, if 
\begin{equation}
\label{eq: I cluster collision} q_i(t) \ne q_j(t), \;\; \forall i \in \I, \forall j \in \I', \text{ and } \forall t \in (t_1, t_2),
\end{equation}
then $E_{\I}(t)= E_{\I}(q(t), \qd(t)) \in H^1((t_1, t_2), \rr)$. In particular, $E_{\I}(t)$ is continuous in $(t_1, t_2)$. 
\end{enumerate}
\end{dfn}

\begin{rem}
Condition (iv) in the above definition shows the energy of a $\I$-cluster is continuous, as long as the masses from the $\I$-cluster do not collide with masses outside of the cluster, even when there are collisions among the masses inside the cluster. This condition was not required by in the original definition of a generalized solution introduced by Bahri and Rabinowitz, see \cite{BR91} and \cite{AZ93}. Our definition here is stronger and follows from \cite[Definition 4.6]{FT04}.
\end{rem}

\begin{prop}
\label{prop: weak solution to generalized} A weak critical point $q \in H^1([T_1, T_2], \X)$ of $\A$ is a generalized solution of \eqref{eq: N body}. 
\end{prop}

\begin{proof}
Let $\en >0$ and $\qn \in H^1(T_1, T_2, \mathcal{X})$ be two sequences satisfying the conditions given in Definition \ref{dfn: weak critical pt Morse index}. Then there is a finite constant $C>0$, such that
\begin{equation}
\ey \int_{t_1}^{t_2} \sum_{i \in \N} m_i |\dot{\qn_i}(t)|^2 \,dt \le \aen(\qn; T) \le C, \;\; \forall n. 
\end{equation}

The fact that $q$ satisfies the first three conditions given in Definition \ref{dfn: generalized solution} is a standard result, for details see \cite{BR91} or \cite{AZ93}. In the following, we will show $q$ also satisfies condition (iv). Given an arbitrary $\I \subset \N$, recall that  
$$ E^{\en}_{\I}(t)=E^{\en}_{\I}(\qn(t), \dot{\qn}(t)) = K_{\I}(\dot{\qn}(t))- U_{\I}(\qn) -\en \U_{\I}(\qn). $$
By a direct computation,
\begin{equation}
\label{eq derivative E I}
\frac{d E^{\en}_{\I}}{dt} = \sum_{i \in \I} \left \langle \frac{\partial U_{\I, \I'}(\qn)}{\partial \qn_i} + \en \frac{\partial \U_{\I, \I'}(\qn)}{\partial \qn_i}  , \dot{\qn_i} \right \rangle. \\
\end{equation}
Let's assume $q$ satisfies \eqref{eq: I cluster collision} for the above $\I$ and an arbitrary sub-interval $(t_1, t_2) \subset [T_1, T_2]$. Since $\qn(t)$ converges to $q(t)$ uniformly on $[0, T]$,
\begin{equation}
\label{eq: bound partial U V}
\left |\frac{\partial U_{\I, \I'}(\qn(t))}{\partial \qn_i} \right |, \left |\frac{\partial \U_{\I, \I'}(\qn(t))}{\partial \qn_i} \right| \le C_1, \;\; \forall t \in [t_1, t_2], \; \forall i \in \I. 
\end{equation}
Here and in the rest of the proof $C_i, i \in \zz^+$, always represents some positive constant independent of $n$. With \eqref{eq derivative E I} and \eqref{eq: bound partial U V}, the Cauchy-Schwarz inequality tells us 
\begin{equation*}
|\dot{E^{\en}_{\I}}(t)|^2 \le C_2 \sum_{i \in \I} m_i |\dot{\qn_i}(t)|^2 \le C_2 \sum_{i \in \N} m_i|\dot{\qn_i}(t)|^2, \; \; \forall t \in [t_1, t_2]. 
\end{equation*}
Then
\begin{equation}
\label{eq: dot E I L2 norm} \int_{t_1}^{t_2} |\dot{E^{\en}_{\I}}(t)|^2 \,dt \le C_2 \int_{t_1}^{t_2} \sum_{i \in \N} m_i |\dot{\qn_i}(t)|^2 \,dt \le C_3.
\end{equation}

Since $U_{\I}$, $\U_{\I}$ are always positive, $E^{\en}_{\I}(t) \le K_{\I}(\dot{\qn}(t))$, $\forall t$. Then 
\begin{equation}
\label{eq: E I L1 norm} \int_{t_1}^{t_2} E^{\en}_{\I}(t) \,dt \le  \ey \int_{t_1}^{t_2} \sum_{i \in \N} m_i |\dot{\qn_i}(t)|^2 \,dt \le C_4.
\end{equation}
Meanwhile by Poincar\'e inequality and \eqref{eq: dot E I L2 norm}, 
\begin{equation}
\label{eq: E I L2 norm} \int_{t_1}^{t_2} |E^{\en}_{\I}(t) - \int_{t_1}^{t_2} E^{\en}_{\I}(s)\,ds|^2 \,dt \le C_5 \int_{t_1}^{t_2} |\dot{E^{\en}_{\I}}|^2 \,dt \le C_6. 
\end{equation}
Then \eqref{eq: E I L1 norm} implies 
$$ \int_{t_1}^{t_2} |E^{\en}_{\I}(t)|^2 \,dt \le C_7. $$
The above inequality and \eqref{eq: dot E I L2 norm} implies $E^{\en}_{\I}(t)$ is a bounded sequence in $H^1([t_1, t_2], \rr)$. After passing to a subsequence, it converges to a $\hat{E}_{\I}(t) \in H^1([t_1, t_2], \rr)$ weakly in $H^1$ norm and strongly in $L^{\infty}$ norm. 

Since $\qd(t)$ and $E_{\I}(t)=E_{\I}(q(t), \qd(t))$ are well defined for any $t \notin \Delta^{-1}(q)$, and 
$$ \qn(t) \to q(t), \;\; \dot{\qn}(t) \to \qd(t), \; \text{ as } n \to \infty, \;\; \forall t \notin \Delta^{-1}(q),$$
we have $E^{\en}_{\I}(t) \to E_{\I}(t)$, for any $t \notin \Delta^{-1}(q)$. As a result, $E_{\I}(t)= \hat{E}_{\I}(t)$, for any $t \in [t_1, t_2] \setminus \Delta^{-1}(q)$. Since $\Delta^{-1}(q)$ is a set of measure zero, $E_{\I}(t) = \hat{E}_{\I}(t)$ as a $H^1$-Sobolev function, and it is continuous in $(t_1, t_2)$.   

\end{proof}

\begin{dfn}
\label{dfn: I cluster coll} Given a path $q \in H^1([T_1, T_2], \X)$ with an $\I$-cluster collision at a moment $t_0 \in (T_1, T_2)$, we say it is \textbf{isolated}, if there is a constant $a>0$ small enough, such that for any $i \in \I$,  
$$ q_i(t) \ne q_j(t), \; \forall t \in [t_0-a, t_0+a] \setminus \{t_0\}, \; \forall j \in \N \setminus \{i \}. $$
\end{dfn}

\begin{prop}
\label{prop: isolated binary coll} Given a weak critical point $q \in H^1([T_1, T_2], \X)$, if there is a binary collision at the moment $t_0 \in (T_1, T_2)$, then it must be isolated.
\end{prop}

\begin{proof}
By Proposition \ref{prop: weak solution to generalized}, $q$ is a generalized solution of \eqref{eq: N body}. In particular it satisfies condition (iv) in Definition \ref{dfn: generalized solution}, then the desired result was already proven in \cite[Corollary 5.12]{FT04}. Once the reader notices that every binary collision is a so called \emph{locally minimal collision} defined in \cite[Definition 5.2]{FT04}.
\end{proof}


\section{Proof of Theorem \ref{thm: fix end not minimizer} and \ref{thm: regularization}} \label{sec: Tanaka} 

To prove the main theorems, three technical lemmas will be needed. We present them as Lemma \ref{lm: lim xin}, \ref{lem: asymp angle} and \ref{lm: lim zt} in this section and postpone their proofs until the next two sections. 

Let $q \in H^1([T_1, T_2], \X)$ be a weak critical point of $\A$ with a binary collision at the moment $t_0 \in (T_1, T_2)$ and $\qn \in H^1([T_1, T_2], \X)$ a sequence of critical points of $\A^{\en}$ satisfying the conditions required in Definition \ref{dfn: weak critical pt Morse index}. Without loss of generality, we may assume such a binary collision is between $m_1$ and $m_2$, i.e.
$$ q_{1}(t_0)= q_{2}(t_0) \ne q_i(t_0), \; \forall i \in \N \setminus \{1, 2 \}. $$
By Proposition \ref{prop: isolated binary coll}, such an binary collision must be isolated, so we may choose an $a>0$ small enough, such that $[t_0-2a, t_0+2a] \subset (T_1, T_2)$ and 
\begin{equation}
q_1(t) \ne q_2(t), \;\; \forall t \in [t_0-2a, t_0 +2a] \setminus \{t_0\};
\end{equation}
\begin{equation} \label{eq: away from binary coll q}
q_i(t) \ne q_j(t), \;\; \forall t \in [t_0 -2a, t_0 +2a], \; \forall i \in \{1, 2\}, \; \forall j \in \N \setminus \{1, 2 \}. 
\end{equation}

For each $n$, we can always find a $t_n \in [t_0-2a, t_0+2a]$ such that 
$$ \dn:=|\qn_1(t_n)-\qn_2(t_n)|= \min \{|\qn_1(t)-\qn_2(t)|: \; t \in [t_0-2a, t_0+2a] \}. $$
Obviously $\dl_n$ converges to $0$, as $n$ goes to infinity. After passing to a subsequence, we may assume the limit of $t_n$ exists. Since $q_1(t_0)=q_2(t_0)$ is an isolated binary collision at the moment $t_0$, $t_n$ must converge to $t_0$. Then for $n$ large enough, $[t_n-a, t_n+a] \subset [t_0-2a, t_0+2a]$. As a result, 
\begin{equation}
 \label{eq: dn} \dn=\min \{|\qn_1(t)-\qn_2(t)|: \; t \in [t_n-a, t_n+a] \}.
 \end{equation} 

By Definition \ref{dfn: weak critical pt Morse index}, $\qn(t)$ converges to $q(t)$ uniformly on $[T_1, T_2]$. According to \eqref{eq: away from binary coll q}, there is constant $C_1>0$ independent of $n$, such that
\begin{equation}
\label{eq: away from bi coll} |\qn_i(t)-\qn_j(t)| \ge C_1, \;\; \forall t \in [t_n-a, t_n+a], \;\;   \forall i \in \{1, 2\}, \; \forall j \in \N \setminus \{1, 2\}.\end{equation}
Let
\begin{equation*}
\ut(\qn): = U(\qn) - \frac{m_1 m_2}{|\qn_1 - \qn_2|^{\al}}, \;\; \fut(\qn):= \U(\qn)- \frac{m_1 m_2}{|\qn_1 - \qn_2|^2}.
\end{equation*}
There are constants $C_2, C_3>0$ independent of $n$, such that 
\begin{equation}
\label{eq: bound partial Ut Vt} \left|\frac{\partial\ut(\qn(t))}{\partial \qn_i} \right|, \left|\frac{\partial \fut(\qn(t))}{\partial \qn_i} \right|\le C_2, \;\; \forall t \in [t_n-a, t_n+a], \; \forall 1 \le i \le 2, 
\end{equation}
\begin{equation}
\label{eq: bound second partial Ut Vt} \left|\frac{\partial^2 \ut(\qn(t))}{\partial\qn_i \partial \qn_i} \right|, \left|\frac{\partial^2 \fut(\qn(t))}{\partial\qn_i \partial \qn_i} \right| \le C_3, \;\; \forall t \in [t_n-a, t_n+a], \; \forall 1 \le i \le 2.  
\end{equation}

We introduce a new function $\etn(t)=(\etn_i(t))_{i=1}^N$ by
\begin{equation} 
\label{eq: etn}  \begin{cases}
\etn_1(t) & = \qn_2(t)-\qn_1(t), \\
\etn_2(t) & = \frac{1}{m_1+m_2}(m_1 \qn_1(t) + m_2 \qn_2(t)) , \\
\etn_i(t) & = \qn_i(t), \; \text{ if } i =3, \dots, N.
\end{cases} 
\end{equation}
By a direct computation, $\etn(t)$ is a solution of 
\begin{align}
\frac{m_1m_2}{m_1+m_2} \ddot{\etn_1} &= -\al m_1 m_2 \frac{\etn_1}{|\etn_1|^{\al+2}}-2m_1 m_2 \frac{\en \etn_1}{|\etn_1|^4} + \frac{\partial\ut(\etn)}{\partial{\etn_1}}+\en \frac{\partial \fut(\etn)}{\partial{\etn_1}}; \label{eq: etn 1}\\
(m_1+m_2) \ddot{\etn_2} &= \frac{\partial \ut(\etn)}{\partial \etn_2} + \en \frac{\partial \fut(\etn)}{\partial \etn_2}; \label{eq: etn 2} \\
m_i \ddot{\etn_i} & = \frac{\partial \ut(\etn)}{\partial \etn_i} + \en \frac{\partial \fut(\etn)}{\partial \etn_i}, \;\; i = 3, \dots, N. \label{eq: etn 3}
\end{align}
This is the Euler-Lagrange equation of the following Lagrangian 
\begin{equation}
\label{eq: L etn} \begin{split}
L(\etn, \dot{\etn})  & = \frac{m_1m_2}{2(m_1+m_2)}|\dot{\etn_1}|^2 + \frac{m_1 +m_2}{2} |\dot{\etn_2}|^2 + \ey \sum_{i=3}^N m_i |\dot{\etn}|^2 \\ 
& +\frac{m_1m_2}{|\etn_1|^{\al}}+ \frac{\en m_1 m_2}{|\etn_1|^2} + \ut(\etn) + \fut(\etn). 
\end{split} 
\end{equation}

To study the behaviors of the solutions as they approach to the binary collision, Tanaka's \emph{blow-up} technique will be used. The precise argument depends on the limit of $\en/\dn^{2-\al}$. After passing to subsequence, we may assume such a limit $\lmd= \lim_{n \to \infty}\en/\dn^{2-\al}$ always exists. Then two different cases need to be considered: \emph{Case 1}, $\lmd \in [0, \infty)$; \emph{Case 2}, $\lmd=\infty$. 

For \emph{Case 1}, we blow up $\etn(t)$ according to
\begin{equation}
\label{eq: blow up xin} \xin(s)= (\xin_i(s))_{i=1}^N: = (\dn^{-1} \etn_i(t(s)))_{i=1}^N, \; \text{ where } t(s) = \dn^{1+\frac{\al}2}s + t_n.
\end{equation}
By changing the time parameter from $t$ to $s$, the time interval $[t_n-a, t_n+a]$ is mapped onto $[-a/\dn^{1+\al/2}, a/\dn^{1+\al/2}]$. Notice that the latter interval converges to $\rr$, as $n$ goes to infinity. Let $'$ denote derivatives with respect $s$, then $\xi_1(s)$ satisfies 
\begin{equation}
\label{eq: xin 1} \begin{split}
\frac{1}{m_1+m_2} (\xin_1)'' & = -\al \frac{\xin_1}{|\xin_1|^{\al+2}}- 2\frac{\en}{\dn^{2-\al}}\frac{\xin_1}{|\xin_1|^4} \\
& + \frac{\dn^{1+\al}}{m_1m_2} \left(\frac{\partial \ut(\etn)}{\partial \etn_1} + \en \frac{\partial \fut(\etn)}{\partial \etn_1}\right).
\end{split}
\end{equation}


\begin{lm}
\label{lm: lim xin} If $\lmd= \lim_{n \to \infty} \en/\dn^{2-\al}$ is finite, then the following results hold. 
\begin{enumerate}
\item[(a).] After passing to a subsequence, $\xin_1(s)$ converges to a $\xi_1(s)$ in $C^2([-\ell, \ell], \rr^d)$, for any $\ell >0$, where $\xi_1(s)$ is a solution of 
\begin{equation}
\label{eq: xih} \frac{1}{m_1+m_2}\xih_1''= -\al \frac{\xih_1}{|\xih_1|^{\al+2}}-2 \frac{\lmd \xih_1}{|\xih_1|^4};  
\end{equation}
\begin{equation}
\label{eq: xih energy} \frac{1}{2(m_1+m_2)}|\xih_1'|^2-\frac{1}{|\xih_1|^{\al}}- \frac{\lmd}{|\xih_1|^2} = 0; 
\end{equation}
\begin{equation}
\label{eq: xih orthogonal} \langle \xih_1(0), \xih_1'(0) \rangle =0. 
\end{equation}
\item[(b).] $\xi_1(t) \in  W(\xi_1), \forall t \in \rr$, where $ W(\xi_1)= \text{span}\{\xi_1(0), \xi_1'(0)\}$ is a $2$-dim subspace of $\rr^d$. Moreover the following limits exist 
$$ \lim_{s \to \pm \infty} |\xi_1(s)| = +\infty, \; \; \lim_{t \to \pm \infty} \frac{\xi_1(s)}{|\xi_1(s)|} = u^{\pm},$$
and 
$$ \angle(u^-, u^+)= 2\pi \frac{\sqrt{1+\lmd}}{2-\al}. $$
For any two unit vectors $u, v \in \rr^d$, $\angle(u,v)$ represents the angle between them. 

\item[(c).] Let $ W^{\perp}(\xi_1)$ be the orthogonal complement of $W(\xi_1)$ in $\rr^d$ and $H(\xi_1)$ the largest subspace of $H^1_0(\rr, W^{\perp}(\xi_1))$, such that
$$ d^2 \ic(\xi_1)(\phi, \phi)<0, \;\; \forall \phi \in H(\xi_1),$$ 
where $\ic$ is the Lagrange action functional corresponding to equation \eqref{eq: xih}:
\begin{equation}
 \label{eq: I} \ic(\hat{\xi}_1) = \int \frac{1}{2(m_1+m_2)}|\hat{\xi}_1'|^2 +\frac{1}{|\hat{\xi}_1|^{\al}} + \frac{\lmd}{|\hat{\xi}_1|^2} \,dt , 
 \end{equation} 
 then $\dim(H(\xi_1)) \ge (d-2)i(\al, \lmd)$, where 
 \begin{equation}
 \label{eq: i al lmd} i(\al, \lmd):= \max \{k \in \zz: \; k < \frac{2\sqrt{1+\lmd}}{2-\al} \}. 
 \end{equation}
 

\end{enumerate}
\end{lm}


\begin{lm}
\label{lem: asymp angle} If $\lmd= \lim_{n \to \infty} \en/\dn^{2-\al}$ is finite, then
$$ \lim_{t \to t_0^{\pm}} \frac{q_2(t)-q_1(t)}{|q_2(t)-q_1(t)|}= u^{\pm},$$
where $u^{\pm}$ are the two unit vectors given in property (b), Lemma \ref{lm: lim xin}. 
\end{lm}

For \emph{Case 2}, we define a blow-up of $\etn$ according to
\begin{equation}
\label{eq: blow up ztn} \ztn(s)=(\ztn_i(s))_{i=1}^N = (\dn^{-1} \etn_i(t(s)))_{i=1}^N, \; \text{ where } t(s)= \en^{-\ey} \dn^2 s + t_n.
\end{equation}
Like the previous case, after changing the time parameter from $t$ to $s$, the time interval $ [t_n-a, t_n+a]$ is mapped onto $ [-a\en^{\ey}/\dn^2, a \en^{\ey}/\dn^2]$, which converges to $\rr$, as $n$ goes to infinity. Again if we let $'$ represents derivatives with respect to $s$, then $\ztn_1(s)$ satisfies 
\begin{equation}
\label{eq: ztn 1} \begin{split}
\frac{1}{m_1 +m_2} (\ztn_1)'' & = -2 \frac{\ztn_1}{|\ztn_1|^4} - \al \frac{\dn^{2-\al}}{\en} \frac{\ztn_1}{|\ztn_1|^{\al+1}} \\
& +\frac{\en}{\dn^3 m_1m_2} \left(\frac{\partial\ut(\etn)}{\partial \etn_1} + \en\frac{\partial\fut(\etn)}{\partial \etn_1}\right).
\end{split}
\end{equation}

\begin{lm}
\label{lm: lim zt} If $\lmd= \lim_{n \to \infty} \en /\dn^{2-\al}=\infty$, then the following results hold. 
\begin{enumerate}
\item[(a).] After passing to a subsequence, $\ztn_1(s)$ converges to a $\zt_1(s)$ in $C^2([-\ell, \ell], \rr^d)$, for any $\ell >0$, where $\zt_1(s)$ is a solution of 
\begin{equation}
\label{eq: zth} \frac{1}{m_1+m_2} \zth_1''= - 2 \frac{\zth_1}{|\zth_1|^4}; 
\end{equation}
\begin{equation}
\label{eq: zth energy} \frac{1}{2(m_1+m_2)} |\zth_1'|^2 - \frac{1}{|\zth_1|^2} = 0; 
\end{equation}
\begin{equation}
\label{eq: zth orthogonal} \langle \zth_1(0), \zth_1'(0) \rangle =0. 
\end{equation}
\item[(b).] $\zt_1(t) \in  W(\zt_1), \forall t \in \rr$, where $ W(\zt_1)= \text{span}\{\zt_1(0), \zt_1'(0)\}$ is a $2$-dim subspace of $\rr^d$, and $\lim_{s \to \pm \infty} |\zt_1(s)|= +\infty$. 

\item[(c).] Let $ W^{\perp}(\zt_1)$ be the orthogonal complement of $W(\zt_1)$ in $\rr^d$ and $H(\zt_1)$ the largest subspace of $H^1_0(\rr, W^{\perp}(\zt_1))$, such that
$$ d^2 \jc(\zt_1)(\phi, \phi)<0, \;\; \forall \psi \in H(\zt_1),$$ 
where $\jc$ is the Lagrange action functional corresponding to equation \eqref{eq: zth}, 
\begin{equation}
 \label{eq: J} \jc(\zth) : = \int \frac{1}{2(m_1+m_2)}|\zth'|^2 + \frac{1}{|\zth|^2} \,ds, 
 \end{equation} 
 then $\dim(H(\zth)) = +\infty$. 
\end{enumerate}
\end{lm}

\begin{prop}
\label{prop: one binary coll} Under the above notation,
\begin{enumerate}
 \item[(a).] if $\lmd= \lim_{n \to \infty} \en/\dn^{2-\al}$ is finite, then 
 \begin{equation}
 \label{eq: one binary coll} \liminf_{n \to \infty} \mo_{t_n-a, t_n+a}(\qn, \aen) \ge (d-2)i(\al, \lmd);
 \end{equation}
  
 \item[(b).] if $\lmd= \lim_{n \to \infty} \en/\dn^{2-\al}= +\infty$, then 
 \begin{equation}
 \liminf_{n \to \infty} \mo_{t_n-a, t_n+a}(\qn, \aen) = +\infty. 
 \end{equation}
  \end{enumerate}  
\end{prop}

\begin{proof} (a). Given an arbitrary $f=(f_i)_{i=1}^N \in H^1_0([t_n-a, t_n+a], \rr^{dN})$ satisfying 
\begin{equation}
\label{eq: f} f_i(t) \equiv 0, \;\; \forall t, \; \text{ and } \; \forall i \ne 1, 
\end{equation}
by \eqref{eq: etn}, for any $c >0$, the center of mass corresponding to the path $\etn(t)+c f(t)$ will always be at the origin. In the following, we shall compute the second variation of $\aen$ at $\etn$ among all $ f \in H^1_0([t_n-a, t_n+a], \rr^{dN})$ satisfying \eqref{eq: f}. By a direct computation,
\begin{equation}
\label{eq: sec var aen} \begin{split}
 d^2\aen(\etn)[f, f] & = \int_{t_n-a}^{t_n+a} \frac{m_1m_2}{m_1+m_2}|\dot{f}_1|^2 -\al m_1m_2 \frac{|f_1|^2}{|\etn_1|^{\al+2}}- 2m_1m_2 \frac{\en |f_1|^2}{|\etn_1|^4} \\
 & + \al(\al+2)m_1m_2 \frac{\langle \etn_1, f_1 \rangle}{|\etn_1|^{\al+4}} + 8m_1m_2\en\frac{\langle \etn_1, f_1 \rangle}{|\etn_1|^6} \\
 & + \langle \frac{\partial^2\ut(\etn)}{\partial \etn_1 \partial \etn_1}f_1, f_1 \rangle+ \en\langle \frac{\partial^2\fut(\etn)}{\partial \etn_1 \partial \etn_1}f_1, f_1 \rangle \,dt.
\end{split}
\end{equation}
Define the linear operators $T_n: H^1_0(\rr, W^{\perp}(\xi_1)) \to H^1_0([t_n-a, t_n+a], \rr^{dN})$ by $f^n(t)=(f^n_i(t))_{i+1}^N = (T_n\phi)(s(t))$, where $s(t)= \dn^{-\frac{2+\al}{2}}(t-t_n)$ and 
$$ f^n_1(t)= \dn \phi(s(t)) \; \text{ and } \; f^n_i(t) \equiv 0, \; \forall i \ne 1. $$
Then
\begin{equation}
\label{eq: sec var aen blow up} \begin{split}
\dn^{-\frac{2-\al}{2}} &d^2 \aen(\etn)[f^n, f^n] = \int_{\text{supp}(\phi)} \frac{m_1m_2}{m_1+m_2}|\phi'|^2- \al m_1m_2 \frac{|\phi|^2}{|\xin_1|^{\al+2}} \\
& -2m_1m_2 \frac{\en}{\dn^{2-\al}} \frac{|\phi|^2}{|\xin_1|^4}  +\al(\al+2)m_1m_2 \frac{\langle \xin_1, \phi \rangle^2}{|\xin_1|^{\al+4}} + 8m_1m_2 \frac{\en}{\dn^{2-\al}}\frac{\langle \xin_1, \phi \rangle^2}{|\xin_1|^6}\\
& + \dn^{2+\al}\left( \langle \frac{\partial^2 \ut(\etn)}{\partial \etn_1 \partial \etn_1}\phi, \phi\rangle^2 + \en \langle \frac{\partial^2 \fut(\etn)}{\partial \etn_1 \partial \etn_1}\phi, \phi\rangle^2 \right) \,ds
\end{split} 
\end{equation}
By \eqref{eq: bound second partial Ut Vt}, there is a constant $C>0$ independent of $n$, such that 
\begin{equation} \label{eq: upper bound partial 2 U V}
\left|\frac{\partial^2 \ut(\etn(t(s)))}{\partial \etn_1 \partial \etn_1}\right|, \left|\frac{\partial^2 \ut(\etn(t(s)))}{\partial \etn_1 \partial \etn_1}\right| \le C, \; \forall s \in [-\frac{a}{\dn^{1+\al/2}}, \frac{a}{\dn^{1+\al/2}}], \; \forall 1 \le i \le 2.
\end{equation}
As $\xi^n_1$ converges to $\xi_1$ in $C^2([\ell, \ell], \rr^d)$, for any $\ell >0$, the right hand side of \eqref{eq: sec var aen blow up} converges to 
$$ m_1m_2 \left(\int_{\text{supp}(\phi)} \frac{1}{m_1+m_2}|\phi'|^2- \al \frac{|\phi|^2}{|\xi_1|^{\al+2}} -2 \lmd \frac{|\phi|^2}{|\xi_1|^4} \,dt \right) =m_1 m_2 \big(d^2 \ic(\xi_1)[\phi, \phi] \big).$$
Then Lemma \ref{lm: lim xin} implies,
\begin{equation*}
\dn^{-\frac{2-\al}{2}} d^2 \aen(\etn)[f^n, f^n] < 0, \; \text{ for } n \text{ large enough, if } \phi \in H(\xi_1). 
\end{equation*}
As $\dim(H(\xi_1)) \ge (d-2)i(\al, \lmd)$, it immediately implies \eqref{eq: one binary coll}. 

(b). If $\lmd$ is infinity, with Lemma \ref{lm: lim zt}, the desired property following from a similar argument as above. We will not repeat it again. 
\end{proof}

With the above results, we can now prove Theorem \ref{thm: fix end not minimizer} and \ref{thm: regularization}. 
\begin{proof}[Proof of Theorem \ref{thm: fix end not minimizer}] By Proposition \ref{prop: isolated binary coll}, if $q(t)$ has a binary collision, then it is isolated. Then $\mathcal{B}(q)$ must be finite.

Assume $q$ has a binary collision at the moment $t_0 \in (T_1, T_2)$, let the corresponding sequences $\en$, $\qn(t)$ and $\dl_n$ be defined as before. If $\lim_{n \to \infty} \en/ \dn^{2-\al} =+\infty$, then property (b) in Proposition \ref{prop: one binary coll} implies $m_{T_1, T_2}^-(q, \A) = +\infty$, which obviously implies \eqref{eq; bound number of binary coll}. 

If the corresponding limit of $\en/\dn^{2-\al}$ is finite for each binary collision, then \eqref{eq; bound number of binary coll} follows from the sub-additivity of the Morse index and property (a) in \ref{prop: one binary coll}. Once we notice $i(\al, \lmd) \ge i(\al)$, for any $\lmd \ge 0$.  
\end{proof}

\begin{proof} [Proof of Theorem \ref{thm: regularization}] Without loss of generality let's assume the binary collision is between $m_1$ and $m_2$ with the sequences $\en$, $\qn(t)$ and $\dn$ defined as before, and $\lmd = \lim_{n \to +\infty}\en/\dn^{2-\al}$. Obviously $\lmd$ must be finite, as otherwise by Proposition \ref{prop: one binary coll}, $m^-_{T_1, T_2}(q, \A)= +\infty$, which is absurd. 

Since $\lmd$ is finite, $\al=1$ and $d=3$, by Proposition \ref{prop: one binary coll}, 
$$ 1 = m^-_{T_1, T_2}(q, \A) \ge i(1, \lmd) = \max \{k \in \zz: \; k < 2 \sqrt{1+\lmd} \}. $$
This implies $\lmd=0$. Then by Lemma \ref{lem: asymp angle}, both limits $\lim_{t \to t_0^{\pm}} \frac{q_2(t)-q_1(t)}{|q_2(t)-q_1(t)|}$ exist and the angle between them is $ 2\pi$, as $\al=1$ and $\lmd=0$. 

\end{proof}


\section{Proof of Lemma \ref{lm: lim xin} and \ref{lm: lim zt}} \label{sec: lem 1}

\begin{proof} [Proof of Lemma \ref{lm: lim xin}] (a). Recall that \eqref{eq: dn} and \eqref{eq: etn} imply, 
$$ |\etn_1(t_n)|= \dn = \min \{ |\etn_1(t)|: \; t \in [t_n-a, t_n+a]\}. $$
Then by the definition of $\xin_1(s)$, 
$$ |\xin_1(s)| \ge |\xin_1(0)| = 1, \; \forall s \in [-a/\dn^{1+\frac{\al}{2}}, a/\dn^{1+\frac{\al}{2}}]. $$
As a result, 
$$ \langle \xin_1(0), (\xin_1)'(0) \rangle =0. $$
Let $\xi_1(s)$ be a solution of \eqref{eq: xih}, with initial condition 
$$ \xi_1(0)= \lim_{n \to \infty} \xin_1(0), \;\; \xi_1'(0) = \lim_{n \to \infty} (\xin_1)'(0).$$
Upon passing to a subsequence, we may assume the above limits always exist. Then $\langle \xi_1(0), \xi_1'(0) \rangle = \lim_{n \to \infty} \langle \xin_1(0), (\xin_1)'(0) \rangle =0$. 

Meanwhile by \eqref{eq: bound partial Ut Vt} and \eqref{eq: etn}, there is a constant $C_1>0$ independent of $n$ with 
$$ \left|\frac{\partial\ut(\etn(t))}{\partial \etn_1} \right|, \left|\frac{\partial \fut(\etn(t))}{\partial \etn_1} \right|\le C_1, \;\; \forall t \in [t_n-a, t_n+a]. $$
This means the second line in equation \eqref{eq: xin 1} converges to zero, which gives us equation \eqref{eq: xih}, as $n$ goes to infinity. Then by the continuous dependence of solutions on initial conditions and coefficients of differential equations, we have $\xin_1(s)$ converges to $\xi_1(s)$ in $C^2([-\ell, \ell], \rr^d)$, for any $\ell >0$. 

To show that $\xi_1(s)$ satisfies \eqref{eq: xih energy}, consider the energy $E^{\en}_{\I}(t)=E^{\en}_{\I}(\qn(t), \dot{\qn}(t))$ and $E_{\I}(t)= E_{\I}(q(t), \qd(t))$ of the $\I$-cluster with $\I=\{1, 2\}$ corresponding to $\qn$ and $q$. By the proof of Proposition \ref{prop: weak solution to generalized}, $E_{\I}(t)$ is continuous on $[t_0-a, t_0+a]$ and $E^{\en}_{\I}(t)$ converges to it under the $L^{\infty}$ norm, after passing to a subsequence. Therefore for $n$ large enough, there is a constant $C_2>0$ independent of $n$, such that 
\begin{equation}
 \label{eq: E en I bound} |E^{\en}_{\I}(t)| \le C_2, \;\; \forall t \in [t_n -a, t_n+a].
 \end{equation} 

Meanwhile with \eqref{eq: etn} and \eqref{eq: blow up xin}, a direct computation shows
\begin{equation}
\label{eq: energy xin I} \frac{1}{2(m_1 +m_2)} |(\xin_1)'|^2 - \frac{1}{|\xin_1|^{\al}} - \frac{\en}{\dn^{2-\al}} \frac{1}{|\xin_1|^2} = \frac{\dn^{2-\al}}{m_1m_2} ( E^{\en}_{\I} - \ey |\dot{\etn_2}|^2)
\end{equation}
To prove that $\xi_1(s)$ satisfies \eqref{eq: xih energy}, it is enough to show the right hand side of the above equation converges to zero, as $n$ goes to infinity. 

To see this, notice that $\etn_2(t)$ converges to $\frac{m_1q_1(t)+m_2q_2(t)}{m_1 + m_2}$, which is the center of mass of $m_1$ and $m_2$. Although they collide at the moment $t_0$, the path of their center of mass, is actually $C^2$ on $[t_0-a, t_0+a]$ (see \cite[Remark 4.10]{FT04}). Hence the convergence of $\etn_2(t)$ holds at least under the $C^1$ norm, and as a result, there is a constant $C_3>0$ independent of $n$, such that 
\begin{equation}
\label{eq: etn bound} 
|\dot{\etn_2}(t)| \le C_3, \;\; \forall t \in [t_n-a, t_n+a]. 
\end{equation}
Now our claim follows directly from \eqref{eq: E en I bound} and \eqref{eq: etn bound}. 

(b). Notice that \eqref{eq: xih} describes the motion of a point mass under the attraction of a isotropic central force. As a result, $\langle \xi_1(0), \xi_1'(0) \rangle =0$ implies  $\xi_1(t) \in W(\xi_1)= \text{span} \{\xi_1(0), \xi'_1(0) \}$, for all $t \in \rr$. By property (a), $\xi_1(s)$ is a collision-free zero energy solution of \eqref{eq: xih}. Then the rest of the property is well known and a detailed proof can be found in \cite[Section 4]{Tn93a}.

(c). Let $\{ \e_i: i=1, \dots, d \}$ be an orthogonal basis of $\rr^d$, such that $W(\xi_1)=\text{span}\{\e_1, \e_2 \}$. Then for any $\varphi \in H^1_0(\rr, \rr)$ and $\e_i$, $i=3, \dots, d$, a simple computation shows
\begin{equation}
 \label{eq; sec var I varphi} 
 d^2 \ic(\xi_1)[\varphi \e_i, \varphi \e_i]  = \int_{-\infty}^{\infty} \frac{ |\varphi'|^2}{2(m_1+m_2)}-\al \frac{|\varphi|^2}{|\xi_1|^{\al+2}}-2 \lmd \frac{|\varphi|^2}{|\xi_1|^4} \,dt
\end{equation} 

Let $\Lmd(\xi_1)$ be the largest subspace of $H^1_0(\rr, \rr)$, such that the value in \eqref{eq; sec var I varphi} is negative for any $\varphi \in \Lmd(\xi_1)$. Using Sturm Comparison Theorem, in \cite[Section 4]{Tn93a}, \cite[Section 4]{Tn93b} and \cite[Proposition 1.1]{Tn94a}, Tanaka showed the dimension of $\Lmd(\xi_1)$ is related to the winding number of $\xi_1(s)$,  $s \in \rr$, in the plane $W(\xi_1)$ with respect to the origin. More precisely 
\begin{equation}
 \label{eq: Lmd i} \text{dim}(\Lmd(\xi_1)) \ge i(\al, \lmd),
 \end{equation}
and this proves property (c). 
\end{proof}

\begin{rem}
\label{rem d=2 no result} Although Tanaka only state \eqref{eq: Lmd i}, a slight modification of his proof should show this is in fact an equality. As a result, if one wants to get a better estimate of the Morse index near a binary collision, one has to compute the Morse index of $\xi_1$ inside $W(\xi_1)$. However a result in \cite[Corollary 5.1]{HY17} by Hu and the author shows this is actually zero. Because of this we believe with Tanaka's approach, one can not get any nontrivial result for the planar $N$-body problem. 
\end{rem}

\begin{proof} [Proof of Lemma \ref{lm: lim zt}] With the results given by Tanaka in \cite[Section 4]{Tn93b}, the lemma can be proven following the same argument given in Lemma \ref{lm: lim xin}. The only difference is the \emph{blow up} should follow \eqref{eq: blow up ztn}, instead of \eqref{eq: blow up xin}, and correspondingly \eqref{eq: energy xin I} needs to be replaced by 
\begin{equation}
 \label{eq: energy ztn I} \frac{1}{2(m_1 + m_2)} |(\ztn_1)'|^2- \frac{1}{|\ztn_1|^2} -\frac{\dn^{2-\al}}{\en} \frac{1}{|\ztn_1|^{\al}} = \frac{\dn^2}{\en m_1m_2} ( E^{\en}_{\I} - \ey |\dot{\etn_2}|^2).
 \end{equation} 
Nevertheless the right hand side of the equation still goes to zero, because $\lmd= +\infty$ implies $\dn^2/\en$ converge to zero, as $n$ goes to infinity. 
\end{proof}

\section{Proof of Lemma \ref{lem: asymp angle}} \label{sec: lm asymp angle}

Our proof follows the approach given by Tanaka in \cite{Tn94a}. First we establish a lemma that corresponds to Lemma 1.3 in \cite{Tn94a}. 

\begin{lm} \label{lm: bound energy and convexity} 
\begin{enumerate}
\item[(a).] There is a constant $C_1>0$ independent of $n$, such that 
\begin{equation}
\label{eq: up bound dot etn 1} |\dot{\etn_1}(t)|^2 \le \frac{2(m_1+m_2)}{|\etn_1(t)|^{\al}} + \frac{2\en(m_1+m_2)}{|\etn_1(t)|^2}+C_1, \;\; \forall t \in [t_n-a, t_n+a].  
\end{equation}
\item[(b).] There is a constant $\ell_0>0$, such that for $n$ large enough, if $t \in [t_n-a, t_n+a]$ and $\etn_1(t) \in B_{\ell_0}$, then $\frac{d^2}{dt^2}|\etn_1(t)|^2 >0$, where $B_{\ell}=\{ x \in \rr^d: \; |x| \le \ell \}$ for any $\ell >0$.
\end{enumerate} 
\end{lm}

\begin{proof}
(a). Let $\I= \{1,2\}$. Recall that $E^{\en}_{\I}(t)= E^{\en}_{\I}(\qn(t), \dot{\qn}(t))$ is the energy of the $\I$-cluster. By \eqref{eq: etn},
\begin{equation}
 \label{eq: energy 12 etn} E^{\en}_{\I}(t) = \frac{m_1m_2}{2(m_1+m_2)}|\dot{\etn_1}|^2+\frac{m_1+m_2}{2} |\dot{\etn_2}|^2-\frac{m_1m_2}{|\etn_1|^{\al}}-\en \frac{m_1m_2}{|\etn_1|^2}.
 \end{equation} 
As a result,
$$ |\dot{\etn_1}(t)|^2 \le \frac{2(m_1+m_2)}{|\etn_1(t)|^{\al}}+2\en \frac{m_1+m_2}{|\etn_1(t)|^2}-\frac{(m_1+m_2)^2}{m_1m_2}|\dot{\etn_2}(t)|^2 + \frac{2(m_1+m_2)}{m_1m_2} E_{\I}^{\en}(t). $$
Then property (a) following from \eqref{eq: E en I bound} and \eqref{eq: etn bound}. 

(b). By a direct computation, 
\begin{equation}
\begin{split}
\ey \frac{d^2}{dt^2} |\etn_1|^2 &= |\dot{\etn_1}|^2+ \langle \etn_1, \ddot{\etn_1} \rangle =  |\dot{\etn_1}|^2-\al \frac{m_1+m_2}{|\etn_1|^{\al}}-2\en \frac{m_1+m_2}{|\etn_1|^2} \\
& \quad + \frac{m_1+m_2}{m_1m_2} \langle \etn_1, \frac{\partial \ut}{\partial \etn_1}+ \en\frac{\partial \fut}{\partial \etn_1} \rangle.  
\end{split}
\end{equation}
Use \eqref{eq: E en I bound} and \eqref{eq: energy 12 etn} , we can find a positive constant $C_2$, such that
$$ \ey \frac{d^2}{dt^2} |\etn_1|^2 \ge (2-\al)\frac{m_1+m_2}{|\etn_1|^{\al}}-\frac{m_1+m_2}{2}|\dot{\etn_2}|^2 + \frac{m_1+m_2}{m_1m_2} \langle \etn_1, \frac{\partial \ut}{\partial \etn_1}+ \en\frac{\partial \fut}{\partial \etn_1} \rangle- C_2,$$
which clearly implies property (b). 

\end{proof}

Again following \cite{Tn94a}, we introduce the following functions
$$ e_n(t)= \sqrt{|\etn_1(t)|^2 |\dot{\etn_1}(t)|^2 - \langle \etn_1(t), \dot{\etn_1}(t)\rangle^2}, \;\; \om(t)= \frac{e_n(t)}{|\etn_1(t)| |\dot{\etn_1}(t)|}. $$
Notice that $\om_n(t)$ is well defined, when $\etn_1(t) \ne 0$ and $\dot{\etn_1}(t) \ne 0$. In particular, $\om_n(t) = \sin (\angle(\etn_1(t)/|\etn_1(t)|, \dot{\etn_1}(t))/|\dot{\etn_1}(t)|)$ and $|\om_n(t)| \le 1$. By \eqref{eq: etn 1},
$$ \frac{d e_n}{dt}= \frac{m_1+m_2}{m_1m_2} \left \langle \frac{|\etn_1|^2 |\dot{\etn_1}|^2-\langle \etn_1, \dot{\etn_1} \rangle \etn}{\sqrt{|\etn_1|^2 |\dot{\etn_1}|^2-\langle \etn_1, \dot{\etn_1} \rangle^2}}, \frac{\partial\ut(\etn)}{\partial{\etn_1}}+\en \frac{\partial \fut(\etn)}{\partial{\etn_1}} \right \rangle. $$
According to \eqref{eq: bound partial Ut Vt}, there are positive constants $C_3, C_4$ independent of $n$, such that 
\begin{equation}
\label{eq: bound dot e} |\dot{e}_n(t)|\le C_3 |\etn_1(t)| \cdot \left|\frac{\partial \ut(\etn_1(t))}{\partial \etn_1}+ \en\frac{\partial \fut(\etn_1(t))}{\partial \etn_1} \right| \le C_4 |\etn_1(t)|, \;\; \forall t \in [t_n-a, t_n+a]. 
\end{equation}

Again by \eqref{eq: etn 1}, a direct computation shows
\begin{equation}
\begin{split}
\frac{d \om_n}{dt} &= \frac{\dot{e}_n}{|\etn_1||\dot{\etn_1}|}-\frac{\om_n}{|\etn_1|^2 |\dot{\etn_1}|^2} \Big( |\dot{\etn_1}|^2 \langle \etn_1, \dot{\etn_1}\rangle + |\etn_1|^2 \langle \dot{\etn_1}, \ddot{\etn_1} \rangle \Big) \\
&=\frac{\dot{e}_n}{|\etn_1||\dot{\etn_1}|} -\frac{m_1+m_2}{m_1m_2} \frac{\om_n}{|\dot{\etn_1}|^2}\langle \dot{\etn_1}, \frac{\partial \ut}{\partial \etn_1}+ \en\frac{\partial \fut}{\partial \etn_1} \rangle \\ 
& \quad - \om_n \frac{\langle \etn_1, \dot{\etn_1} \rangle}{|\etn_1|^2 |\dot{\etn_1}|^2} \Big(|\dot{\etn_1}|^2-\al \frac{m_1+m_2}{|\etn_1|^{\al}}-2\en\frac{m_1+m_2}{|\etn_1|^2} \Big) \\
\end{split}
\end{equation}
Then \eqref{eq: up bound dot etn 1} and \eqref{eq: bound dot e} implies,
\begin{equation} \label{eq: om dot}
\begin{split}
\dot{\om}_n(t) & \le \frac{C_5}{|\dot{\etn_1}|} \left|\frac{\partial \ut}{\partial \etn_1}+ \en\frac{\partial \fut}{\partial \etn_1} \right|- \om_n \frac{\langle \etn_1, \dot{\etn_1} \rangle}{|\etn_1|^2 |\dot{\etn_1}|^2} \Big( (2-\al)\frac{m_1+m_2}{|\etn_1|^{\al}}-C_6 \Big) \\
& \le \left(C_5|\etn_1|\left|\frac{\partial \ut}{\partial \etn_1}+ \en\frac{\partial \fut}{\partial \etn_1}\right| -\om_n \sqrt{1-\om_n^2} \Big((2-\al)\frac{m_1+m_2}{|\etn_1|^{\al}}-C_6\Big) \right) \frac{1}{|\etn_1||\dot{\etn_1}|}, 
\end{split}
\end{equation}
where $C_5$ is positive constant and $C_6$ is a constant, whose sign depends on the sign of $\langle \etn_1, \dot{\etn_1} \rangle$), independent of $t \in [t_n-a, t_n+a]$ and $n$. 

With Lemma \ref{lm: bound energy and convexity} and \eqref{eq: om dot} (this corresponds to ($1.8$) in \cite{Tn94a}), the next result can be proven following the argument given in Proposition 1.4 and 1.5 in \cite{Tn94a} line by line, and we will not repeat it here.
\begin{lm}
 \label{lm: eta 1 cauchy} Let $\ell_0$ be the constant given in Lemma \ref{lm: bound energy and convexity}, for any $b>0$ small enough, there exist constants $0 < \ell_2 <\ell_0 <\ell_1$, such that when $n$ is large enough, for any 
$$ t_n < t < t^* < t_n +a, \; \text{ or } \; t_n-a < t < t^* < t_n, $$
if $\etn_1(t)$ and $\etn_1(t^*) \in B_{\ell_2} \setminus B_{\dn \ell_1} $, then $ \Big| \etn_1(t)/|\etn_1(t)| - \etn_1(t^*)/|\etn_1(t^*)| \Big| < b.$
 \end{lm} 

Now we give a proof of Lemma \ref{lem: asymp angle}.
\begin{proof} [Proof of Lemma \ref{lem: asymp angle}] By property (c) in Lemma \ref{lm: lim xin}, $\lim_{s \to +\infty} \xi_1(s)/|\xi_1(s)|= u^+$. Fix an arbitrary small $b>0$, let $\ell_1, \ell_2$ be the constants given in  Lemma \ref{lm: eta 1 cauchy}. We can choose an $s^*>0$ large enough, such that for $n$ large enough, 
$$ \ell_1 < |\xi_1(s^*)| < \dl_n^{-1} \ell_2, \; \text{ and } \;  \left| \frac{\xi_1(s^*)}{|\xi_1(s^*)|}- u^+ \right| < b. $$  
By Lemma \ref{lm: lim xin}, the same inequalities hold for $\xin_1(s^*)$, when $n$ is large enough. 

Let $t^*= t_n+\dn^{1+\frac{\al}{2}}s^*$, by equation \eqref{eq: blow up xin}, 
$$ \etn_1(t^*) \in B_{\ell_2} \setminus B_{\dn \ell_1}, \text{ and } |\etn_1(t^*)/|\etn_1(t^*)| -u^+| < b. $$
Since $\ell_2 < \ell_0$, for any $t \in (t^*, t_n +a)$, we claim if $\etn_1(t) \in B_{\ell_2}$, then $|\etn_1(t)| > \dn \ell_1.$ Indeed this follows from the fact $d |\etn_1(t_n)|/dt=0$ and property (b) in Lemma \ref{lm: bound energy and convexity}. 

As a result, for any  $t \in (t^*, t_n +a)$,  $\etn_1(t) \in B_{\ell_2}$ implies $\etn_1(t) \in B_{\ell_2} \setminus B_{\dn \ell_1}$. Then by Lemma \ref{lm: eta 1 cauchy}, 
$$ \left| \frac{\etn_1(t)}{|\etn_1(t)|} - u^+ \right| \le \left| \frac{\etn_1(t)}{|\etn_1(t)|}-\frac{\etn_1(t^*)}{|\etn_1(t^*)|} \right| + \left| \frac{\etn_1(t^*)}{|\etn_1(t^*)|}-u^+ \right| \le 2b. $$
This means 
$$ (\{ \etn_1(t): \; t \in (t_n+ \dn^{1+\frac{\al}{2}}s^*, t_n+a) \} \cap B_{\ell_2}) \subset (\{x \in B_{\ell_2}: \; |x/x|-u^+| < 2b \} \cup B_{\dn \ell_1}). $$
Recall that when $n$ goes to infinity, $t_n$ converges to $t_0$, $\dn$ converges to zero, and $\etn_1(t)$ converges uniformly to $q_2(t)-q_1(t)$. Then 
$$ (\{q_2(t)-q_1(t): \; t \in (t_0, t_0+a) \} \cap B_{l_2}) \subset \{ x \in B_{l_2}:\; |x/x|-u^+| < 2b \}. $$
Since the above result hold for any $b>0$ small enough, we get 
$$ \lim_{t \to +\infty} \frac{q_2(t)-q_1(t)}{|q_2(t)-q_1(t)|} = u^+. $$
A similar argument shows 
$$ \lim_{t \to -\infty}  \frac{q_2(t)-q_1(t)}{|q_2(t)-q_1(t)|} = u^-. $$
\end{proof}

\emph{Acknowledgements.} 
The author thanks Alain Chenciner for a careful reading of an early draft of the paper and Richard Montgomery for pointing out a mistake. Valuable discussions with Vivina Barutello, Jacques F\'ejoz and Xijun Hu are also appreciated. The main part of the work was done when the author was a postdoc and a visitor at Ceremade in University of Paris-Dauphine, IMCCE in the Paris Observatory and School of Mathematics in Shandong University. He thanks the hospitality of these institutes and the financial support of FSMP and NSFC(No.11425105).

\bibliographystyle{abbrv}
\bibliography{RefMountPass}

\end{document}